\documentclass{amsart}

\usepackage{amsmath,amsxtra,amsthm,amssymb}
\newtheorem{theorem}{Theorem}[section]
\newtheorem{lemma}[theorem]{Lemma}
\newtheorem{proposition}[theorem]{Proposition}

\newtheorem{definition}[theorem]{Definition}
\newtheorem{remark}[theorem]{Remark}

\newcommand{\p}{\mathfrak{p}}
\newcommand{\bp}{\overline{\p}}
\newcommand{\QQ}{\mathbb{Q}}
\newcommand{\ZZ}{\mathbb{Z}}

\newcommand{\Zp}{\ZZ_p}
\newcommand{\Qp}{\QQ_p}

\newcommand{\Gal}{\operatorname{Gal}}

\begin{document}

\title{Factorisation of two-variable $p$-adic $L$-functions}
\author{Antonio Lei}
\address{Department of Mathematics and Statistics\\
Burnside Hall\\
McGill University\\
Montreal QC\\
Canada H3A 0B9}
\email{antonio.lei@mcgill.ca}
\thanks{The author is supported by a CRM-ISM fellowship.}

\maketitle
\begin{abstract}
Let $f$ be a modular form which is non-ordinary at $p$. Kim and Loeffler have recently constructed two-variable $p$-adic $L$-functions associated to $f$. In the case where $a_p=0$, they showed that, as in the one-variable case, Pollack's plus and minus splitting applies to these new objects. In this short note, we show that such a splitting can be generalised to the case where $a_p\ne0$ using Sprung's logarithmic matrix.
\end{abstract}
\section{$p$-adic logarithmic matrices}
We first review the theory of Sprung's factorisation of one-variable $p$-adic $L$-functions in \cite{sprung09,sprung12}, which is a generalisation of Pollack's work \cite{pollack03}.

Let $f=\sum_{n\ge1} a_n q^n$ be a normalised eigen-newform of weight $2$ and level $N$. Fix an odd\footnote{Our results in fact hold for $p=2$. Since the interpolation formulae of $p$-adic $L$-functions are slightly different from the other cases, we assume $p\ne2$ for notational simplicity.} prime $p$ that does not divide $N$ and $v_p(a_p)\ne0$. Here, $v_p$ is the $p$-adic valuation. Let $\alpha$ and $\beta$ be the two roots to
\[
X^2-a_pX+\epsilon(p)p=0
\]
with $r=v_p(\alpha)$ and $s=v_p(\beta)$. Note in particular that $0<r,s<1$.

Let $G$ be a one-dimensional $p$-adic Lie group, which is of the form $\Delta\times \langle\gamma_p\rangle$,
where $\Delta$ is a finite abelian group and $\langle\gamma_p\rangle\cong\Zp$.  Let $F$ be a finite extension of $\Qp$ that contains $\mu_{|\Delta|}$, $a_n$ and $\epsilon(n)$ for all $n\ge1$. For a real number $u\ge0$, we define $ D^{(u)}(G,F)$ for the set of distributions $\mu$ on $G$ which are of the form
\[
\sum_{n\ge0}\sum_{\sigma\in\Delta}c_{\sigma,n}\sigma (\gamma_p-1)^n
\]
where $c_{\sigma,n}\in F$ and $\sup_n \frac{|c_{\sigma,n}|_p}{n^u}<\infty$ for all $\sigma\in\Delta$ (here $|\cdot|_p$ denotes the $p$-adic norm with $|p|_p=p^{-1}$). Let $X=\gamma_p-1$. If $\eta$ is a character on $\Delta$, we write $e_\eta\mu$ for the $\eta$-isotypical component of $\mu$, namely, the power series
\[
\sum_{n\ge0}\sum_{\sigma\in\Delta}c_{\sigma,n}\eta(\sigma) (\gamma_p-1)^n
\in F[[X]].
\]
For $\mu_1\in D^{(u)}(\langle\gamma_p\rangle,F)$ and $\mu_2\in D^{(u)}(G,F)$, we say that $\mu_1$ divides $\mu_2$ over $D^{(u)}(G,F)$ if $\mu_1$ divides all isotypical components of  $\mu_2$ as elements in $F[[X]]$.

\begin{definition}
We say that $(\mu_\alpha,\mu_\beta)\in D^{(r)}(G,F)\oplus D^{(s)}(G,F)$ is a pair of interpolating functions for $f$ if for all non-trivial characters $\omega$ on $G$ that send $\gamma_p$ to a primitive $p^{n-1}$-st root of unity for some $n\ge1$, there exists a constant $C_\omega\in \overline{F}$ such that
\[
\mu_\alpha(\omega)=\alpha^{-n}C_\omega\quad\text{and}\quad \mu_\beta(\omega)=\beta^{-n}C_\omega.
\]
\end{definition}

\begin{remark}
The $p$-adic $L$-functions $L_\alpha, L_\beta$ of Amice-V\'{e}lu \cite{amicevelu75} and Vi{\v{s}}ik \cite{visik76} associated to $f$ satisfy the property stated above, with $G$ being the Galois group $\Gal(\QQ(\mu_{p^\infty})/\QQ)$ and $C_\omega$ being the algebraic part of the complex $L$-value $L(f,\omega^{-1},1)$ multiplied by some fudge factor.
\end{remark}

\begin{definition}
A  matrix $M_p=\begin{pmatrix}
m_{1,1}& m_{1,2}\\
m_{2,1}&m_{2,2}
\end{pmatrix}$  with $m_{1,1},m_{2,1}\in D^{(r)}(\langle\gamma_p\rangle,F)$ and $m_{1,2},m_{2,2}\in D^{(s)}(\langle\gamma_p\rangle,F)$, is called a $p$-adic logarithmic matrix associated to $f$ if $\det(M_p)$ is, up to a constant in $F^\times$, equal to $\log_p(\gamma_p)/(\gamma_p-1)$, and  $\det(M_p)$ divides both
\[
m_{2,2}\mu_\alpha-m_{2,1}\mu_\beta\quad\text{and}\quad -m_{1,2}\mu_\alpha+m_{1,1}\mu_\beta
\]
over $ D^{(1)}(G,F)$ for all interpolating functions $\mu_\alpha$, $\mu_\beta$ for $f$. 
\end{definition}

\begin{lemma}\label{lem:factorisation}
Let  $\mu_\alpha$, $\mu_\beta$ be a pair of interpolating functions  for $f$. If $M_p$ is a $p$-adic logarithm  matrix associated to $f$, then there exist  $\mu_{\#},\mu_{\flat}\in D^{(0)}(G,F)$ such that
\[
\begin{pmatrix}
\mu_{\alpha}&\mu_{\beta}
\end{pmatrix}
=\begin{pmatrix}
\mu_{\#}&\mu_{\flat}
\end{pmatrix}M_p.
\]
\end{lemma}
\begin{proof}
Let
\[
\mu_{\#}:=\frac{m_{2,2}\mu_\alpha-m_{2,1}\mu_\beta}{\det(M_p)}\quad\text{and}\quad \mu_{\flat}:=\frac{-m_{1,2}\mu_\alpha+m_{1,1}\mu_\beta}{\det(M_p)}.
\]
By definition, the numerators lie inside $D^{(1)}(G,F)$ and the coefficients of $\det(M_p)$ have the same growth rate as those of $\log_p(\gamma_p)$, so $\mu_{\#}$ and $\mu_{\flat}$ lie inside $D^{(0)}(G,F)$. The factorisation follows from the fact that
\[
\begin{pmatrix}
m_{2,2}&-m_{1,2}\\ -m_{2,1}&m_{1,1}
\end{pmatrix}M_p=\begin{pmatrix}\det(M_p)&0\\0&\det(M_p)\end{pmatrix}.
\]
\end{proof}

We now recall the construction of Sprung's canonical $p$-adic logarithmic matrix associated to $f$.

Let $C_n=\begin{pmatrix}
a_p&1\\-\epsilon(p)\Phi_{p^n}(\gamma_p)&0
\end{pmatrix}$, where $\Phi_{p^n}$ denotes the $p^n$-th cyclotomic polynomial for $n\ge1$, $C=\begin{pmatrix}a_p&1\\-\epsilon(p)p&0\end{pmatrix}$ and $A=\begin{pmatrix}
-1&-1\\\beta&\alpha
\end{pmatrix}$. Define 
\[
M_p^{(n)}:=C_1\cdots C_n C^{-n-2}A.
\]

\begin{theorem}[Sprung]\label{thm:matrix}
The entries of the sequence of matrices $M_p^{(n)}$ converge (under the standard sup-norm on $p$-adic power series) in $D^{(1)}(\langle\gamma_p\rangle,F)$ as $n\rightarrow\infty$ and the limit $\displaystyle\lim_{n\rightarrow\infty}M_p^{(n)}$ is a $p$-adic logarithmic matrix associated to $f$.
\end{theorem}
\begin{proof}
We only sketch our proof here since this is merely a slight generalisation of Sprung's results in \cite{sprung12,sprung09}.

Since $C_{n+1}\equiv C\mod (X+1)^{p^n}-1$, we have
\[
M_p^{(n+1)}\equiv M_p^{(n)}\mod (X+1)^{p^n}-1.
\]
Note that
\[
A^{-1}CA=\begin{pmatrix}
\alpha&0\\0&\beta
\end{pmatrix},
\]
which implies that
\begin{equation}\label{eq:diagonal}
C^{-n-2}A=A\begin{pmatrix}
\alpha^{-n-2}&0\\0&\beta^{-n-2}\end{pmatrix}=
\begin{pmatrix}
-\alpha^{-n-2}&-\beta^{-n-2}\\
\beta\alpha^{-n-2}&\alpha\beta^{-n-2}
\end{pmatrix}.
\end{equation}
Since all the entries in $C_1\cdots C_n$ are integrals, the coefficients of the first (respectively second) row of $M_p^{(n)}$ grow like $O(p^{-r})$ (respectively $O(p^{-s})$) as $n\rightarrow\infty$. Therefore, by \cite[\S1.2.1]{perrinriou94}, the entries of the first (respectively second) row of $M_p^{(n)}$ converge to elements in $D^{(r)}(\langle\gamma_p\rangle,F)$ (respectively $D^{(s)}(\langle\gamma_p\rangle,F)$).

Let $M_p=\begin{pmatrix}
m_{1,1}& m_{1,2}\\
m_{2,1}&m_{2,2}
\end{pmatrix}$ be the limit $\displaystyle\lim_{n\rightarrow\infty}M_p^{(n)}$. If $\omega$ is a character that sends $\gamma_p$ to a primitive $p^{n-1}$-st root of unity, then
\[
M_p(\omega)=C_1(\omega)\cdots C_{n-1}(\omega)C^{-n-1}A.
\]
Note that $C_{n-1}(\omega)=\begin{pmatrix}a_p&1\\0&0\end{pmatrix}$, so from \eqref{eq:diagonal}, we see that there exist two constants $A_\omega,B_\omega\in\overline{F}$ such that
\[
M_p(\omega)=\begin{pmatrix}a_pA_\omega&A_\omega\\ a_pB_\omega& B_\omega\end{pmatrix}\begin{pmatrix}
-\alpha^{-n-1}&-\beta^{-n-1}\\
\beta\alpha^{-n-1}&\alpha\beta^{-n-1}
\end{pmatrix}=
\begin{pmatrix}
-\alpha^{-n}A_\omega&-\beta^{-n}A_\omega\\
-\alpha^{-n}B_\omega&-\beta^{-n}B_\omega
\end{pmatrix}.
\]
In particular, if $\mu_\alpha,\mu_\beta$ is a pair of interpolating functions  for $f$,
\[
m_{2,2}(\omega)\mu_\alpha(\omega)-m_{2,1}(\omega)\mu_\beta(\omega)=-m_{1,2}(\omega)\mu_\alpha(\omega)+m_{1,1}(\omega)\mu_\beta(\omega)=0.
\]

Finally, by \cite[Remark~2.19]{sprung12}, $\det(M_p)=\frac{\log_p(1+X)}{X}\times\frac{\beta-\alpha}{(\alpha\beta)^2}$, hence the result.
\end{proof}

\begin{remark}
Similar logarithmic matrices have been constructed in \cite{leiloefflerzerbes10} using the theory of Wach modules, but they are not canonical.
\end{remark}


\section{Two-variable $p$-adic $L$-functions}

\subsection{Setup for two-variable distributions}

We now fix an imaginary quadratic field $K$ in which $p$ splits into $\p\bp$. If $\mathfrak{I}$ is an ideal of $K$, we write $G_{\mathfrak{I}}$ for the Ray class group of $K$ modulo $\mathfrak{I}$. Define
\[
G_{p^\infty}=\varprojlim G_{p^n},\quad G_{\p^\infty}=\varprojlim G_{\p^n},\quad G_{{\bp}^\infty}=\varprojlim G_{{\bp}^n}.
\]
These are the Galois groups of the Ray class fields $K(p^\infty)$, $K(\p^\infty)$ and $K({\bp}^\infty)$ respectively. Fix topological generators $\gamma_\p$ and $\gamma_{\bp}$ of the $\Zp$-parts of $G_{\p^\infty}$ and $G_{{\bp}^\infty}$ respectively. We have an isomorphism
\[
G_{p^\infty}\cong \Delta\times\langle\gamma_\p\rangle\times\langle\gamma_{\bp}\rangle,
\]
where $\Delta$ is a finite abelian group. Let $X=\gamma_\p-1$ and $Y=\gamma_{\bp}-1$.  For  real numbers $u,v\ge0$, we define $D^{(u,v)}(G_{p^\infty},F)$ for the set of distributions of $G_{p^\infty}$ which are of the form
\[
\sum_{i,j\ge0}\sum_{\sigma\in\Delta}c_{\sigma,i,j}\sigma X^iY^j,
\] 
where $c_{\sigma,i,j}\in F$ and $\sup_{i,j}\frac{|c_{\sigma,i,j}|_p}{i^uj^v}<\infty$ for all $\sigma\in\Delta$. On identifying each $\Delta$-isotypical component of $\mu$ with a power series in $X$ and $Y$, we have the notion of divisibility as in the one-dimensional case. We define the operators $\partial_\p$ and $\partial_{\bp}$ to be the partial derivatives $\frac{\partial}{\partial X}$ and $\frac{\partial}{\partial Y}$ respectively.

For $\star\in\{\p,\bp\}$, we let $\Omega_\star$ be the set of characters on $G_{p^\infty}$ with conductor $(\star)^n$ for some integer $n\ge1$.

Let $\mu\in D^{(u,v)}(G_{p^\infty},F)$ where $u,v\ge 0$. If $\omega_{\p}\in\Omega_\p$, we define a distribution $\mu^{(\omega_\p)}$ by
\[
\mu^{(\omega_\p)}(\omega_{\bp})=\mu(\omega_\p\omega_{\bp}).
\]
\begin{lemma}
The distribution $\mu^{(\omega_\p)}$ lies inside $D^{(v)}(G_{{\bp}^\infty},F')$, where $F'$ is the extension $F(\omega_{\p}(\gamma_\p))$.
\end{lemma}
\begin{proof}
We need to show that
\begin{equation}\label{eq:required}
\sup_j\frac{|\sum_i c_{\sigma,i,j}(\omega_{\p}(\gamma_\p)-1)^i|_p}{j^v}<\infty
\end{equation}
for all $\sigma\in\Delta$.
But $v_p(\omega_\p(\gamma_\p)-1)>0$, so there exists a constant $C$, such that
\[
|(\omega_{\p}(\gamma_\p)-1)^i|_p\le \frac{C}{i^u}
\]
for all $i$. Hence,
\[
\frac{|c_{\sigma,i,j}(\omega_{\p}(\gamma_\p)-1)^i|_p}{j^v}\le C\times\frac{|c_{\sigma,i,j}|_p}{i^uj^v},
\]
which implies \eqref{eq:required}.
\end{proof}

Similarly, for $\omega_{\bp}\in\Omega_{\bp}$, we may define a distribution $\mu^{(\omega_{\bp})}\in D^{(u)}(G_{{\p}^\infty},F')$, where $F'=F(\omega_{\bp}(\gamma_{\bp}))$.

\subsection{Sprung-type factorisation}
Let $L_{\alpha,\alpha}$, $L_{\alpha,\beta}$, $L_{\beta,\alpha}$, $L_{\beta,\beta}$ be the two-variable $p$-adic $L$-functions constructed in \cite{loeffler13} (also \cite{kim11}). By \cite[Theorem~4.7]{loeffler13}, $L_{\star,\bullet}$ is an element of $D^{(v_p(\star),v_p(\bullet))}(G_{p^\infty},F)$ for $\star,\bullet\in\{\alpha,\beta\}$. Moreover, if $\omega$ is a character on $G_{p^\infty}$ of conductor $\p^{n_\p}{\bp}^{n_{\bp}}$ with $n_\p,n_{\bp}\ge1$, we have
\begin{eqnarray}
L_{\alpha,\alpha}(\omega)&=&\alpha^{-n_\p}\alpha^{-n_{\bp}}C_\omega\label{eq:interpolation1}\\
L_{\alpha,\beta}(\omega)&=&\alpha^{-n_\p}\beta^{-n_{\bp}}C_\omega\label{eq:interpolation2}\\
L_{\beta,\alpha}(\omega)&=&\beta^{-n_\p}\alpha^{-n_{\bp}}C_\omega\label{eq:interpolation3}\\
L_{\beta,\beta}(\omega)&=&\beta^{-n_\p}\beta^{-n_{\bp}}C_\omega\label{eq:interpolation4}
\end{eqnarray}
for some $C_\omega\in \overline{F}$ that is independent of $\alpha$ and $\beta$.

Let $M_p$ be the logarithmic matrix given by Theorem~\ref{thm:matrix}. On replacing $\gamma_p$ by $\gamma_\p$ and $\gamma_{\bp}$ respectively, we have two logarithmic matrices $M_\p=\begin{pmatrix}
m_{1,1}^{\p}& m_{1,2}^{\p}\\
m_{2,1}^{\p}&m_{2,2}^{\p}
\end{pmatrix}$ and $M_{\bp}=\begin{pmatrix}
m_{1,1}^{\bp}& m_{1,2}^{\bp}\\
m_{2,1}^{\bp}&m_{2,2}^{\bp}
\end{pmatrix}$ defined over $D^{(1)}(\langle\gamma_\p\rangle,F)$ and $D^{(1)}(\langle\gamma_{\bp}\rangle,F)$ respectively. 

Our goal is to prove the following generalisation of \cite[Corollary~5.4]{loeffler13}.

\begin{theorem}\label{thm:main}
There exist  $L_{\#,\#},L_{\flat,\#},L_{\#,\flat},L_{\flat,\flat}\in D^{(0,0)}(G_{p^\infty},F)$ such that 
 \[
\begin{pmatrix}
L_{\alpha,\alpha}&L_{\beta,\alpha}&L_{\alpha,\beta}&L_{\beta,\beta}
\end{pmatrix}
=\begin{pmatrix}
L_{\#,\#}&L_{\flat,\#}&L_{\#,\flat}&L_{\flat,\flat}
\end{pmatrix}M_{\p}\otimes M_{\bp}.
\]
\end{theorem}

We shall prove this theorem in two steps, namely, to show that we can first factor out $M_\p$, then $M_{\bp}$.

\begin{proposition}\label{prop:factorisation1}
For $\star\in\{\alpha,\beta\}$, there exist $L_{\#,\star},L_{\flat,\star}\in D^{(0,v_p(\star))}(G_{p^\infty},F)$ such that
\begin{equation}\label{eq:factorisation1}
\begin{pmatrix}
L_{\alpha,\star}&L_{\beta,\star}
\end{pmatrix}
=\begin{pmatrix}
L_{\#,\star}&L_{\flat,\star}
\end{pmatrix}M_\p.
\end{equation}
\end{proposition}
\begin{proof}
We take $\star=\alpha$ (since the proof for the case $\star=\beta$ is identical). Let $\omega_\p\in\Omega_\p$ and $\omega_{\bp}\in\Omega_{\bp}$ and write $\omega=\omega_\p\omega_{\bp}$.

By \eqref{eq:interpolation1} and \eqref{eq:interpolation3}, $L_{\alpha,\alpha}^{(\omega_{\bp})}$ and $L_{\beta,\alpha}^{(\omega_{\bp})}$ is a pair of interpolating functions for $f$. In particular, $\det(M_\p)$ divides both 
\[
m_{2,2}^\p L_{\alpha,\alpha}^{(\omega_{\bp})}-m_{2,1}^\p L_{\beta,\alpha}^{(\omega_{\bp})}\quad\text{and}\quad -m_{1,2}^\p L_{\alpha,\alpha}^{(\omega_{\bp})}+m_{1,1}^\p L_{\beta,\alpha}^{(\omega_{\bp})}
\]
over $D^{(1)}(G_{\p^\infty},F)$. Therefore, the distributions
\[
m_{2,2}^\p L_{\alpha,\alpha}-m_{2,1}^\p L_{\beta,\alpha}\quad\text{and}\quad -m_{1,2}^\p L_{\alpha,\alpha}+m_{1,1}^\p L_{\beta,\alpha} 
\]
vanish at all characters of the form $\omega=\omega_\p\omega_{\bp}$. This implies that
\[
\left(m_{2,2}^\p L_{\alpha,\alpha}-m_{2,1}^\p L_{\beta,\alpha}\right)^{(\omega_\p)}=\left( -m_{1,2}^\p L_{\alpha,\alpha}+m_{1,1}^\p L_{\beta,\alpha}\right)^{(\omega_\p)}=0
\]
since these two distributions lie inside $D^{(r)}(G_{{\bp}^\infty},F')$ for some $F'$, with $r<1$, and they vanish at an infinite number of characters for each of their isotypical components. Hence, $\det(M_\p)$ divides 
\[
m_{2,2}^\p L_{\alpha,\alpha}-m_{2,1}^\p L_{\beta,\alpha}\quad\text{and}\quad -m_{1,2}^\p L_{\alpha,\alpha}+m_{1,1}^\p L_{\beta,\alpha} 
\]
over $D^{(1,r)}(G_{p^\infty},F)$. Let
\[
L_{\#,\alpha}:=\frac{m_{2,2}^\p L_{\alpha,\alpha}-m_{2,1}^\p L_{\beta,\alpha}}{\det(M_\p)}\quad\text{and}\quad L_{\flat,\alpha}:=\frac{-m_{1,2}^\p L_{\alpha,\alpha}+m_{1,1}^\p L_{\beta,\alpha}}{\det(M_\p)}.
\]
We may then conclude as in the proof of Lemma~\ref{lem:factorisation}.
\end{proof}

\begin{lemma}\label{lem:partial}
Let $\omega$ be a character of $G_{p^\infty}$ of conductor $\p^{n_\p}{\bp}^{n_{\bp}}$ with $n_\p,n_{\bp}\ge1$. There exist constants $D_\omega$ and $E_\omega$ in $\overline{F}$ such that
\begin{eqnarray*}
\partial_\p L_{\alpha,\alpha}(\omega)=\alpha^{-n_{\bp}}D_\omega,&&\partial_\p L_{\alpha,\beta}(\omega) =  \beta^{-n_{\bp}}D_\omega,\\
\partial_\p L_{\beta,\alpha}(\omega)=\alpha^{-n_{\bp}}E_\omega,&&\partial_\p L_{\beta,\beta}(\omega) =  \beta^{-n_{\bp}}E_\omega.
\end{eqnarray*}
\end{lemma}
\begin{proof}
We only prove the result concerning $\partial_\p L_{\alpha,\alpha}$ and $\partial_\p L_{\alpha,\beta}$. Fix an $\omega_{\bp}\in\Omega_{\bp}$. By \eqref{eq:factorisation1} and \eqref{eq:factorisation2}, we have
\[
 \beta^{n_{\bp}}L_{\alpha,\beta}^{(\omega_{\bp})}(\omega_\p)=\alpha^{n_{\bp}}L_{\alpha,\alpha}^{(\omega_{\bp})}(\omega_\p)
\]
for all $\omega_\p\in \Omega_\p$. But $L_{\alpha,\beta}^{(\omega_{\bp})}, L_{\alpha,\alpha}^{(\omega_{\bp})}\in D^{(r)}(G_{\p^\infty},F')$ for some $F'$. As $r<1$, this implies that
\[
 \beta^{n_{\bp}}L_{\alpha,\beta}^{(\omega_{\bp})}=\alpha^{n_{\bp}}L_{\alpha,\alpha}^{(\omega_{\bp})}.
\]
In particular, their derivatives agree, that is
\[
\beta^{n_{\bp}}\partial_\p L_{\alpha,\beta}^{(\omega_{\bp})}=\alpha^{n_{\bp}}\partial_\p L_{\alpha,\alpha}^{(\omega_{\bp})}.
\]
But for a general $\mu\in D^{(r,s)}(G_{p^\infty},F)$, we have
\[
\partial\left(\mu^{(\omega_{\bp})}\right)(\omega_\p)= \partial_\p \mu(\omega_\p\omega_{\bp})
\]
for all $\omega_\p\in\Omega_\p$, hence
\[
\beta^{n_{\bp}}\partial_\p L_{\alpha,\beta}(\omega)=\alpha^{n_{\bp}}\partial_\p L_{\alpha,\alpha}(\omega)
\]
as required.
\end{proof}

\begin{proposition}
For $\star\in\{\#,\flat\}$, there exist $L_{\star,\#},L_{\star,\flat}\in D^{(0,0)}(G_{p^\infty},F)$ such that
\begin{equation}\label{eq:factorisation2}
\begin{pmatrix}
L_{\star,\alpha}&L_{\star,\beta}
\end{pmatrix}
=\begin{pmatrix}
L_{\star,\#}&L_{\star,\flat}
\end{pmatrix}M_{\bp}.
\end{equation}
\end{proposition}
\begin{proof}
Let us prove the proposition for $\star=\#$.  Let $\omega_\p\in\Omega_\p$ and $\omega_{\bp}\in\Omega_{\bp}$ and write $\omega=\omega_\p\omega_{\bp}$. Recall that
\[
L_{\#,\bullet}\det(M_\p)=m_{2,2}^\p L_{\alpha,\bullet}-m_{2,1}^\p L_{\beta,\bullet}
\]
for $\bullet\in\{\alpha,\beta\}$. Since $\det(M_\p)$ is, up to a non-zero constant in $F^{\times}$, equal to $\log_p(1+X)/X$, we have $\det(M_p)(\omega_\p)=0$ and  $\partial_\p\det(M_\p)(\omega_\p)\ne0$. On taking partial derivatives, Lemma~\ref{lem:partial} together with \eqref{eq:interpolation1}-\eqref{eq:interpolation4} imply that
\[
L_{\#,\bullet}(\omega)=(\bullet)^{-n_{\bp}}\frac{K_\omega}{\partial_\p\det(M_\p)(\omega_\p)},
\]
where $K_\omega$ is the constant
\[
m_{2,2}^\p(\omega_\p)D_\omega+\partial_\p m_{2,2}^\p(\omega_\p) \alpha^{-n_\p}C_\omega-m_{2,1}^\p(\omega_\p) E_\omega-\partial_\p m_{2,1}^\p(\omega_\p)\beta^{-n_\p} C_\omega.
\]
In particular, we see that $L_{\#,\alpha}^{(\omega_\p)}$ and $L_{\#,\beta}^{(\omega_\p)}$ is a pair of interpolating functions for $f$, so we may proceed as in the proof of Proposition~\ref{prop:factorisation1} (with the roles of $\p$ and $\bp$ swapped).
\end{proof}

Combining the factorisations \eqref{eq:factorisation1} and \eqref{eq:factorisation2}, we obtain Theorem~\ref{thm:main}.

\subsection*{Acknowledgement}
The author would like to thank David Loeffler and Henri Darmon for their very helpful comments.

\bibliographystyle{amsalpha}
\bibliography{../references}

\end{document}